\documentclass[a4paper, 15pt]{amsart}

\usepackage{amscd,amsthm,amsfonts,amssymb,amsmath,esint}
\usepackage[colorlinks=true]{hyperref}
\usepackage{fullpage}
\usepackage[all]{xy}
\usepackage[margin=1.5in]{geometry}
\usepackage{mathrsfs}
\usepackage{tikz-cd}

\newcommand{\msb}{\mathscr{B}}

\newcommand{\mst}{\mathscr{T}}

    \newcommand{\BC}{{\mathbb {C}}} 
     \newcommand{\BF}{{\mathbb {F}}}

    \newcommand{\BQ}{{\mathbb {Q}}}

     \newcommand{\BZ}{{\mathbb {Z}}}

    \newcommand{\CO}{{\mathcal {O}}}

    \newcommand{\fa}{{\mathfrak{a}}}

    \newcommand{\fm}{{\mathfrak{m}}} 
     \newcommand{\fp}{{\mathfrak{p}}}
    \newcommand{\fq}{{\mathfrak{q}}}

    \newcommand{\fw}{{\mathfrak{w}}}

    \newcommand{\fC}{{\mathfrak{C}}} 
     \newcommand{\fF}{{\mathfrak{F}}}

    \newcommand{\fM}{{\mathfrak{M}}} 
     \newcommand{\fP}{{\mathfrak{P}}}

    \newcommand{\fW}{{\mathfrak{W}}}

    \newcommand{\Cl}{{\mathrm{Cl}}}

    \newcommand{\End}{{\mathrm{End}}}

    \newcommand{\Gal}{{\mathrm{Gal}}} 
    
    \newcommand{\Hom}{{\mathrm{Hom}}}

    \newcommand{\Ker}{{\mathrm{Ker}}}

    \newcommand{\ord}{{\mathrm{ord}}}

    \renewcommand{\mod}{\ \mathrm{mod}\ }

    \newcommand{\Sel}{{\mathrm{Sel}}}

    \font\cyr=wncyr10

    \newcommand{\Sha}{\hbox{\cyr X}}

    \theoremstyle{plain}
    \newtheorem{thm}{Theorem}[section] \newtheorem{cor}[thm]{Corollary}
    \newtheorem{lem}[thm]{Lemma}  \newtheorem{prop}[thm]{Proposition}

\theoremstyle{remark} 
\theoremstyle{remark} 
\theoremstyle{remark} 

    \numberwithin{equation}{section}

\begin{document}

\title {On the $\lambda$-invariant of Selmer groups arising from certain quadratic twists of Gross curves}

\author{Jianing Li}



\maketitle

\begin{abstract}
Let $q$ be a prime with $q \equiv 7 \mod 8$, and let $K=\BQ(\sqrt{-q})$. Then $2$ splits in $K$, and  we write $\fp$  for either of the primes $K$ above $2$. Let $K_\infty$ be the unique $\BZ_2$-extension of $K$ unramified outside $\fp$ with $n$-th layer $K_n$. For certain quadratic and biquadratic extensions $\fF/K$, we prove a simple exact formula for the $\lambda$-invariant of the Galois group of the maximal abelian 2-extension unramified outside $\fp$ of the field $\fF_\infty = \fF K_\infty$. Equivalently, our result determines the exact $\BZ_2$-corank of certain Selmer groups over
$\fF_\infty$ of a large family of quadratic twists of the higher dimensional abelian variety with complex multiplication, which is the restriction of scalars to $K$
of the Gross curve with complex multiplication defined over the Hilbert class field of $K$. We also discuss computations of the associated Selmer groups over $K_n$ in the case when the $\lambda$-invariant is equal to $1$.
\end{abstract}

\section{Introduction}

The aim of the present paper is to prove a simple exact formula for the $\lambda$-invariant of a certain classical Iwasawa module at the prime $p=2$. Our result can be viewed
as an analogue of a result of Kida \cite{K} and Ferrero \cite{Fer}. However, there is a fundamental difference in that, unlike Kida's and Ferrero's work, the Iwasawa module we consider 
has a simple interpretation in terms of classical infinite descent theory on a family of abelian varieties with complex multiplication arising from certain quadratic twists of the Gross
\cite{Gross1} family of elliptic curves with complex multiplication.

\medskip

We first explain our result in terms of classical Iwasawa theory. Let $q$ be any prime with $q \equiv 7 \mod 8$, and let $K=\BQ(\sqrt{-q})$. Write $\CO_K$ for the ring of integers of $K$. 
Then $2$ splits in $K$, say $2\CO_K=\fp\fp^*$. By class field theory, there is a unique $\BZ_2$-extension  $K_\infty/K$ which is unramified outside $\fp$, and we write $K_n$ for the unique intermediate field with $[K_n:K] = 2^n$, and $\CO_{K_n}$ for the ring of integers of $K_n$.  Throughout, we assume that $\fp$  corresponds to the embedding $\iota_\fp: K\to \BQ_2$ determined by $\iota_\fp(\sqrt{-q})\equiv 1\bmod 4\BZ_2$. For each prime $\fw$ of $K$, we write $\ord_\fw$ for the normalized additive valuation at $\fw$. Also, put $\fq = \sqrt{-q}\CO_K$.

%If $W$ is any algebraic extension of $K$, we define $L(W)$ to be the maximal unramified abelian $2$-extension of $W$, and $M(W)$ to be the maximal abelian $2$-extension of $W$, which is unramified outside the primes of $W$ lying above $\fp$.
\medskip

In all that follows, $R$ will denote any element of $\CO_K$, which will always be assumed to satisfy the following:-\\

\noindent {\em Twisting Hypothesis}\label{1}  We have $R\equiv 1 \mod 4\CO_K$, $(\fq, R) = 1$, and $\ord_{\fw}(R)$ is odd for each prime $\fw$ of $K$ dividing $R$.\\

\noindent We then define
\begin{equation}\label{b}
F=K(\sqrt{-\sqrt{-q}R}), \quad F'=K(\sqrt{ \sqrt{-q}R}),\quad D=K(\sqrt{-1}) \quad  \text{ and } \quad J=FF'=FD.
\end{equation}
For any finite extension $T/K$, let
\begin{equation}
T_n=TK_n, \quad T_\infty=TK_\infty.
\end{equation}
Further, let $M(T_\infty)$ be the maximal abelian $2$-extension of $T_\infty$, which is unramified outside the primes of $T_\infty$ lying above $\fp$. Then the main result of this paper is as follows.

%To state our main result, we let
%\begin{equation}\label{eq:def_r}
%s= \text{the number of primes of } K_\infty \text{  dividing } (\sqrt{-q}R).
%\end{equation}

\begin{thm}\label{thm:main}
Let $s_\infty(R)$ be the number of primes of $K_\infty$ dividing $\sqrt{-q}R$, where $R$ satisfies the above Twisting Hypothesis. Then (i) $X(F_\infty)$ is a free finitely generated $\BZ_2$-module of exact rank $s_\infty(R) - 1$; and (ii) if $q\equiv 7\bmod 16$, then $X(F'_\infty)$ (resp. $X(J_\infty)$) is a free finitely generated $\BZ_2$-module of exact rank $s_\infty(R)$ (resp. $2s_\infty(R)-1$).
\end{thm}

 \noindent We point out that the number $s_\infty(R)$ is finite and can be computed explicitly as follows. Let $h$ be the class number of $K$, which is odd.
Suppose $\fw$ is a prime of $K$ distinct from $\fp$. Then $\fw^h=(w)$ is principal. Adjust the sign of $w$ so that $\iota_\fp(w)\equiv 1\bmod 4\BZ_2$. Then, by class field theory (see the proof of Lemma 2.3 of \cite{CLL} or of Lemma 3.3 of \cite{CL}), the number of primes of $K_\infty$ dividing $\fw$ is 
\begin{equation*}
2^{\ord_2( \iota_\fp(w) -1)-2}.
\end{equation*}
It follows that if $q\equiv 7\bmod 16$ and $R=1$, we have $s_\infty(R) = 1$ whence $X(J_\infty)$ is a free $\BZ_2$-module of rank $1$. This recovers \cite[Theorem 1.1]{CLL}. It was proven in \cite{L1} that $X(F_\infty)\neq 0$ when $q\equiv 15\bmod 16$ and $R=1$, but the methods there are different from those of the present paper. The fact that $X(F_\infty)$ is a free finitely generated $\BZ_2$-module of $\BZ_2$-rank at most $s_\infty(R)-1$ has been proven when $q =7$ in \cite{CC}, and the argument extends easily to the more general case considered here. In particular, it can be proved in this way that $X(D_\infty)=0$ when $q\equiv 7\bmod 16$ (see \cite[Proposition 2.2]{CLL} or Proposition~\ref{prop:Xbound}).

%What is new in the present paper is to show that $s_\infty(R)-1$ is also a lower bound for this rank. We do this by constructing $s_\infty(R)-1$ independent quadratic extensions of $F_\infty$ which are contained in $M(F_\infty)$. Our proof makes essential use of a generalization of a classical lemma of Chevalley (see \cite{Gras} or \cite{LY}).

\medskip

As we shall explain in \S4, Theorem \ref{thm:main} has an equivalent formulation in terms of a certain Selmer group arising from infinite descent on the $h$-dimensional abelian
variety which is the twist by the quadratic extension $K(\sqrt{R})/K$ of the abelian variety $B/K$; here $B/K$ is the restriction of scalars from the Hilbert class field $H$ of $K$ to $K$ of the Gross elliptic curve $A/H$ with complex multiplication by $\CO_K$ \cite{Gross1}. It seems that no such clear cut result giving the exact $\BZ_2$-corank of 
Selmer groups of a large family of quadratic twists of an abelian variety over a $\BZ_2$-extension of the base field was known previously. In \S5, shall we discuss more on the case $s_\infty(R)=2$ where $R$ is a square free rational integer satisfying the Twisting Hyptheosis.

\medskip

The author is deeply grateful to John Coates for discussing this problem and for his careful reading and polishing of the manuscript. Thanks  also goes to Yongxiong Li for helpful discussions and comments.

\section{Prelimaries}

For this section alone, we consider a more general situation than in the rest of the paper. Take $p$ to be an arbitrary prime number, and take $K$ to be any imaginary quadratic field in which $p$ splits. We fix one of the primes $\fp$ of $K$ above lying above $p$. By class field theory, there is a unique $\BZ_p$-extension  $K_\infty$ of $K$ which is unramified outside $\fp$. If $W/K$ is any finite extension, we define $W_\infty = W K_\infty$. Let $M(W_\infty)$ be the maximal abelian $p$-extension of $W_\infty$, which is unramified outside the primes lying above $\fp$, and put $X(W_\infty) = \Gal(M(W_\infty)/W_\infty)$. Again by class field theory, we have:-
\begin{lem}\label{lem:ram_basic1}
There are only finitely many primes of $W_\infty$ lying above each prime of $W$.
\end{lem}
\noindent  For each finite extension $\fW/W$, we define $S_\infty(\fW/W)$ to be the set of primes of $W_\infty$, which do not lie above $\fp$, and which are ramified in $\fW_\infty$. Of course, $S_\infty(\fW/W)$ is a finite set by Lemma~\ref{lem:ram_basic1}. Put
\begin{equation}\label{eq:defs}
s_\infty(\fW/W) = \# (S_\infty(\fW/W)).
\end{equation}

Assume that $\fW/W$ is a cyclic extension of degree $p$ with $\fW\not\subset W_\infty$. Thus $\fW_\infty/W_\infty$ is cyclic of degree $p$, and we put $\Delta = \Gal(\fW_\infty/W_\infty)$, As usual, $X(\fW_\infty)$ has its natural structure as a module over the group ring $\BZ_p[\Delta]$. 
The maximal ideal of the local ring $\BZ_p[\Delta]$ is $(p,\delta-1)$, and its residue field is the finite field $\BF_p$ with $p$ elements; here $\delta$ is any generator of $\Delta$. If $X$ is a $\BZ_p[\Delta]$-module, put $X_{\Delta} = X/(\delta-1)X$.

%If $X$ is a finitely generated $\BZ_p$-module, we let $\rk_p(X)$ denote the $\BF_p$-dimension of $X/pX$. Then $\rk_p(X)$ is the minimal number of generators of $X$ as a $\BZ_p$-module. 

\begin{lem}\label{lem:prelim}
Let $\fW/W$ be a cyclic extension of degree $p$ with $\fW\not\subset F_\infty$. Then (i) $X(\fW_\infty)$ is finitely generated as a $\BZ_p$-module if and only if $X(W_\infty)$ is finitely generated as a $\BZ_p$-module, (ii) if $X(W_\infty) = 0$, then $(X(\fW_\infty))_{\Delta}$ is an $\BF_p$-vector space of dimension at most $s_\infty(\fW/W)-1$, and (iii) the restriction map induces an isomorphism $(X(\fW_\infty) \otimes_{\BZ_p} \BQ_p)_{\Delta} \cong X(W_\infty)\otimes_{\BZ_p}\BQ_p$.
\end{lem}

\begin{proof}
If $X(\fW_\infty)$ is a finitely generated $\BZ_p$-module, so is its quotient $\Gal( \fW_\infty M(W_\infty)/\fW_\infty )$. Since $[\fW_\infty:W_\infty]$ is finite, it follows that the abelian group $\Gal(\fW_\infty M(W_\infty)/W_\infty)$ is finitely generated over $\BZ_p$,  and hence so is its quotient $X(W_\infty)$. Conversely, assume that $X(W_\infty)$ is finitely generated over $\BZ_p$, and let $T$ denote the maximal abelian extension of $W_\infty$ contained in $M(\fW_\infty)$. Then we have an isomorphism 
\begin{equation}\label{1}
X(\fW_\infty)_{\Delta} \cong \Gal(T/\fW_\infty).
\end{equation} 
Now $M(W_\infty)$ is a subfield of $T$, and $\Gal(T/M(W_\infty))$ is generated by the inertial subgroups of the  primes in $S_\infty(\fW/W)$ in the extension $T/W_\infty$. Such an inertial subgroup is of order $p$ since we are assuming that $\fW_\infty/W_\infty$ is of degree $p$. Thus $\Gal(T/M(W_\infty))$ is an $\BF_p$-vector space, and
\begin{equation}\label{eq:prelim1.1}
\dim_{\BF_p} \Gal(T/M(W_\infty)) \leq s_\infty(\fW/W).
\end{equation}
Since we are assuming that $X(W_\infty)$ is finitely generated over $\BZ_p$, and the set $S_\infty(\fW/W)$ is finite, it follows easily that $\Gal(T/\fW_\infty)$ is finitely generated over $\BZ_p$. Thus, in view of \eqref{1}, $X(\fW_\infty)/(p,\delta-1) X(\fW_\infty) $ is a finite dimensional $\BF_p$-vector space. Since $X(\fW_\infty)$ is a compact $\BZ_p[\Delta]$-module, it follows from Nakayama's lemma that $X(\fW_\infty)$ is a finitely generated  $\BZ_p[\Delta]$-module, whence it is also finitely generated as a $\BZ_p$-module. This completes the proof of (i).  For (ii), when $X(W_\infty)=0$, \eqref{eq:prelim1.1} becomes $\dim _{\BF_p}\Gal(T/W_\infty)\leq s_\infty(\fW/W)$, whence $\dim_{\BF_p} \Gal(T/\fW_\infty)\leq s_\infty(\fW/W)-1$. Recalling \eqref{1}, the assertion (ii) follows.
To prove (iii), we first note that \eqref{1} implies that we have the exact sequence
\begin{equation}\label{2}
1 \to  X(\fW_\infty)_{\Delta} \to \Gal(T/W_\infty) \to \Gal(\fW_\infty/W_\infty) \to 1.
\end{equation}
We also have the exact sequence
\begin{equation}\label{3}
1 \to \Gal(T/M(W_\infty)) \to \Gal(T/W_\infty) \to X(W_\infty) \to 1.
\end{equation}
Since $\Gal(\fW_\infty/W_\infty)$ and $\Gal(T/M(W_\infty))$ are both finite, it follows that 
$$ 
(X(\fW_\infty)_{\Delta}) \otimes_{\BZ_p} \BQ_p \cong X(W_\infty) \otimes_{\BZ_p} \BQ_p.
$$
 This completes the proof.
\end{proof}

\begin{cor}
If $\fW/W$ is a finite Galois extension of $p$-power degree, then $X(\fW_\infty)$ is a finitely generated $\BZ_p$-module if and only if $X(W_\infty)$ is a finitely generated $\BZ_p$-module.
\end{cor}

\noindent This is immediate because any finite $p$-group is solvable.

\begin{lem}\label{lem:naka1}
Let $\Delta$ be a cyclic group of order $p$. Let Y be a free $\BZ_p$-module of finite rank, which is also a $\Delta$-module. Assume
that $Y_\Delta = (\BZ/p\BZ)^r$ for some integer $r\geq 0$. Then 
(i) $N_{\Delta} (Y)=0$, where $N_{\Delta} =\sum_{\tau \in \Delta} \tau$, whence $Y$ is a $\BZ_p[\Delta]/(N_{\Delta})$-module, and
(ii)  $Y$ is a quotient of $\BZ^{(p-1)r}_{p}$ for any prime $p$, and $Y\cong \BZ^r_2$ when $p=2$.
\end{lem}

\begin{proof}
The proof, which the authors attribute to Sharifi, is entirely similar to the case $p=2$ given in \cite[Lemma 2.8]{CL}.
\end{proof}

\noindent Combining Lemmas \ref{lem:prelim} and \ref{lem:naka1}, we immediately obtain:-
\begin{lem}\label{lem:prelim2}
Let $\fW/W$ be a cyclic extension of degree $p$. Assume that $X(W_\infty)=0$, and that $X(\fW_\infty)$ is a torsion free $\BZ_p$-module. Then $X(\fW_\infty)$ is a finitely generated $\BZ_p$-module of rank at most $(p-1)(s_\infty(\fW/W)-1)$.  
\end{lem}

For a finitely generated $\BZ_p$-module $Y$, we write $t_p (Y)$ for its $\BZ_p$-rank. In our applications, we only need the following lemma for the prime $p=2$. 

\begin{lem}\label{lem:prelim3}
Assume that $p=2$, and let $W$ be any finite extension of $K$ such that  $X(W_\infty)$ is a finitely generated $\BZ_2$-module. Take $\fM$ to be any Galois extension of $W$ with $\Gal(\fM/W) \cong \Gal(\fM_\infty/W_\infty)\cong (\BZ/2\BZ)^2$. Then 
\begin{equation}
t_2( X(\fM_\infty)) + 2 t_2(X(W_\infty)) = \sum_{i=1}^{3} t_2(X(\fW_{i,\infty})),
 \end{equation}
where $\fW_1,\fW_2,\fW_3$ are the three quadratic extensions of $W$ contained in $\fM$.
\end{lem}
\begin{proof}  Since we are assuming that $X(W_\infty)$ is a finitely generated $\BZ_2$-module, the same is true for $X(\fM_\infty)$ because $\Gal(\fM/W)$ is a $2$-group. Let $V = X(\fM_\infty)\otimes_{\BZ_2} \BQ_2$, so that the $\BQ_2$-dimension of $V$ is equal to $t_2(X(\fM_\infty))$. Put $\Omega = \Gal(\fM_\infty/W_\infty)$. Then 
\begin{equation*}\label{eq:prelim2}
V = \oplus_{\chi} V^{e_\chi},
\end{equation*}
where the sum runs over the characters $\chi$ of $\Omega$, and  $e_\chi = \sum_{\sigma\in \Omega} \chi(\sigma) \sigma$. Thus we must compute the $\BQ_2$-dimension of
each $V^{e_\chi}$. Note that if $Y$ is any vector space over $\BQ_2$ and $g$ is the non-identity element of a group $\Delta$ of order $2$ acting on Y, then plainly
\begin{equation}\label{4}
Y = Y^{1+g} \oplus Y^{1-g}, \, \, (Y)_\Delta = Y^{1+g}.
\end{equation}
Suppose first that $\chi$ is the trivial character of $\Omega$, so that $e_\chi = (1+\sigma)(1+\tau)$, where, of course, $\sigma$ and $\tau$ are distinct elements $\Omega$ of exact order $2$. By  the remark \ref {4} and (iii) of Lemma \ref{lem:prelim} applied to the group $\langle\sigma\rangle$ of order 2, we conclude that $V^{1+\sigma} = X(T_\infty)\otimes_{\BZ_2}\BQ_2$, where $T_\infty$ is the fixed field of $\sigma$. Thus
\begin{equation}\label{5}
V^{\chi} = (X(T_\infty)\otimes_{\BZ_2}\BQ_2)^{1+\tau} = X(W_\infty)\otimes_{\BZ_2}\BQ_2;
\end{equation}
here the second equality follows from applying the remark \ref{4} and (iii) of Lemma \ref{lem:prelim} to the quadratic extension $T_\infty/W_\infty$. Hence in this case the $\BQ_2$-dimension of $V^{\chi}$ is equal to the $\BQ_2$-dimension of $X(W_\infty)\otimes_{\BZ_2}\BQ_2$.
Suppose next that $\chi$ is a nontrivial character of $\Omega$, say with $\chi(\sigma) = 1$ and $\chi(\tau) = -1$, so that $e_\chi = (1+\sigma)(1-\tau)$.  If $T$ is now the fixed field of $\Ker (\chi )= \{1,\sigma\}$, we again conclude from the remark \ref{4} and (iii) of Lemma \ref{lem:prelim} that $V^{1+\sigma} = X(T_\infty)\otimes_{\BZ_2}\BQ_2$, whence
\begin{equation}\label{eq:prelim3}
V^{e_\chi} \cong (X(T_\infty)\otimes_{\BZ_2}{\BQ}_2)^{1-\tau}. 
\end{equation}
But applying (iii) of Lemma \ref{lem:prelim} to the extension $T_\infty/W_\infty$, we conclude that
$$
X(T_\infty)\otimes_{\BZ_2}{\BQ}_2/(X(T_\infty)\otimes_{\BZ_2}{\BQ}_2)^{1-\tau} = X(W_\infty)\otimes_{\BZ_2}\BQ_2,
$$
whence it follows that the $\BQ_2$-dimension of $V^{e_\chi}$ is equal to the $\BQ_2$-dimension of $X(T_\infty)\otimes_{\BZ_2}{\BQ}_2$ minus the $\BQ_2$-dimension of
$X(W_\infty)\otimes_{\BZ_2}\BQ_2$. Since this holds for each of the three non-trivial characters of $\Omega$, the proof of 
Lemma~\ref{lem:prelim3} is now complete.
\end{proof}

\section{Proof of Theorem \ref{thm:main}}

We now return to the situation discuused in $\S 1$. Thus from now on $K = \BQ(\sqrt{-q})$, where $q$ is any prime with $q \equiv 7 \mod 8$, and again the fields $F, F', D, J$ are defined by \eqref{b}. As always, $\fp$ will denote one of  the primes of $K$ lying above 2, and we again fix the sign of $\sqrt{-q}$ so that $\sqrt{-q} \equiv 1 \mod \fp^2$.  Recall that, for each finite exrension $W/K$,  the integer $s_\infty(W/K)$ is defined by \eqref{eq:defs} in $\S 2$. Again $R$ will denote an arbitrary element of $\CO_K$ satisfying the Twisting Hypothesis, and $s_\infty(R)$ will denote the number of primes of $K_\infty$ dividing $\sqrt{-q}R$. 
%We now give the proof of Theorem \ref{thm:main}. 

\begin{lem}\label{lem:ram1} 
We have $s_\infty(F/K) = s_\infty(R)$ for all primes $q \equiv 7 \mod 8$. Moreover, If $q\equiv 7\bmod 16$, we also have $s_\infty(F'/K)=s_\infty(R) +1$ and $s_\infty(D/K)=1$.
\end{lem}

\begin{proof}
The first assertion is clear because the primes of $K_\infty$, not lying above $\fp$, which ramify in $F_\infty$ are precisely the primes dividing $\sqrt{-q}R$. Note that $\fp^*$ is ramified in $F'/K$ and in $D/K$ by our choice of sign. Then the second assertion follows from the fact that $\fp^*$ is inert in $K_\infty$ when $q \equiv 7\bmod 16$ (see \cite[Proposition 2.4]{CLL}).
\end{proof}

%\begin{proof}
%Since $h$ is odd, $\fp$ is totally ramified in $K_\infty$. It follows that there are precisely $s_\infty(R)+1$ primes of $K_\infty$ dividing $\fp \sqrt{-q}R$. Note that $F/K$ is unramified at $\fp^*$, since by our choice, the embedding $\iota_{\fp^*}: K\to \BQ_2$ corresponding to $\fp^*$ must satisfy $\iota_{\fp^*}(\sqrt{-q})\equiv 3\bmod 4\BZ_2$, whence $\iota_{\fp^*}(-\sqrt{-q}R)\equiv 1\bmod 4$ by the definition of $R$. Thus $F_\infty/K_\infty$ is unramified outside the primes dividing $\fp \sqrt{-q} R$.
%
%We next claim that $\fp$ is totally ramified in $F_\infty/K$. It clearly suffices to show that $\fp$ is totally ramified in $F_1/K$. We have $K_1=K(\sqrt{\pi})$, where $\pi$ is some generator of $\fp^h$ (see \cite[Corollary 2.5]{CLL}), whence $\iota_\fp(\pi) = 2^h u$ for some $u\in \BZ^\times_2$. Also $\iota_\fp(-\sqrt{-q}R)\equiv 1 \bmod 4\BZ_2$. Since $F_1= KF_1$ and $h$ is odd, it follows that 
%\begin{equation*}
%F_1 K_{\fp} = \BQ_2(\sqrt{2u}, \sqrt{ -v })  \quad   \text{ with } v=1 \text{ or }5. 
%\end{equation*}
%This is a totally ramified biquadratic extensions of $\BQ_2$, since clearly all three quadratic subfields are ramified over $\BQ_2$. Thus $F_1/K$ is totally ramified at $\fp$. On the other hand, the $s_\infty(R)$ primes of $K_\infty$ dividing $\sqrt{-q}R$ are plainly ramified in $F_\infty$, and the proof is complete. \end{proof}

\begin{prop}[Choi-Coates\cite{CC}]\label{prop:Xbound}
(i) $X(F_\infty)$ is a free finitely generated $\BZ_2$-module of rank at most $s_\infty(R)-1$, and (ii) if $q\equiv 7\bmod 16$, then $X(D_\infty)=0$ and $X(F'_\infty)$ is a free finitely generated $\BZ_2$-module of rank at most $s_\infty(R)$.
\end{prop}

\begin{proof}
It is easy to see $X(K_\infty)=0$. (for example see \cite[Lemma 3.2]{CL}). Then by Lemma~\ref{lem:prelim} $X(F_\infty)$, $X(D_\infty)$ and $X(F'_\infty)$ are all finitely generated over $\BZ_2$. Moreover, by Greenberg's theorem \cite{GR}, they are torsion-free whence are free of finite rank. The assertions then follow from Lemma~\ref{lem:ram1} and Lemma~\ref{lem:prelim2}. 
\end{proof}

%\begin{proof} We only briefly sketch the proof, which hinges on Nakayama's lemma, since it is entirely similar to the proof for $q = 7$ given in \cite{CC} (see also the first part of \S{3} of \cite{CL}). Put $G = \Gal(F_\infty/K)$, so that $G = \Gamma \times \Delta$, where $\Gamma = \Gal(F_\infty/F) \cong \BZ_2$ and $\Delta$ is cyclic of order $2$. Then one shows by a simple ramification argument (see Prop. 2.5 of \cite{CC}) that $(X(F_\infty))_\Delta$ is a $\BF_2$-vector space of dimension at most $s_\infty(R) - 1$, whence, viewing $X(F_\infty)$
%as a module over the group ring $\BZ_2[\Delta]$, it follows from Nakayama's lemma, that $X(F_\infty)$ is finitely generated  over $\BZ_2$. Moreover, an old theorem of Greenberg \cite{GR}
%shows that $X(F_\infty)$ must then be torsion free, whence by a remark of Sharifi, the fact that $(X(F_\infty))_\Delta = \BF_2^r$ with $\Delta$ of order 2 implies that $X(F_\infty)$ is free of rank $r$, completing this sketch of the proof.
%\end{proof}

\begin{prop}\label{prop:ray1}
Let $\fC_n$ be the ray class group of $K_n$ modulo $(\fp^*)^2\CO_{K_n}$. Then $\fC_n$ has odd order for all $n\geq 0$.
\end{prop}

\noindent Before giving the proof, we recall Chevalley's formula for ray class groups (see \cite{Gras} or \cite[\S 4]{LY}, for example). Suppose that $L/k$ is any cyclic extension of number fields with Galois group $G$, and write $\CO_L$ (resp. $\CO_k$) for the ring of integers of $L$ (resp. $k$).  For our application, we may assume that $k$ has no real places.
If $\fm$ is an integral ideal in $k$, let $\Cl^\fm_{L}$ (resp. $\Cl^\fm_k$) denote the ray class group of $L$ (resp. $k$) modulo $\fm \CO_L$ (resp. $\fm$). Clearly $\Cl^\fm_L$ is a $G$-module, and $(\Cl^{\fm}_L)^G$ will denote, as always, its $G$-invariants. For a prime $v$ (resp. $w$) of $k$ (resp. $L$), write $\CO_v$ (resp. $\CO_w$) for the ring of integers of the completion $k_v$ (resp. $L_w$). Let $L^\fm$ denote the set of all $a \in L^\times$ such that $\ord_w(a-1) \geq \ord_w(\fm)$ for all places $w$ of $L$ dividing $\fm$. Let $\CO^\times_{k,\fm}$ be the subgroup of the unit group $\CO_k^\times$ of $k$ consisting of all units $a$ with $\ord_v(a-1) \geq \ord_v(\fm)$ at all places $v$ of $k$ dividing $\fm$. Then Chevalley's formula with respect to $L/k$ and the ideal $\fm$ is as follows:-
\begin{equation}\label{eq:chevalley}
\frac{\# ((\Cl^{\fm}_L)^G)}{\# (\Cl^\fm_k) } = \frac{\prod\limits_{v\mid \fm} [1+\fm \CO_v: N_{L_w/k_v}(1+\fm \CO_w)] \cdot \prod\limits_{v\nmid \fm} e_v}{[L:k] [\CO^\times_{k,\fm}: \CO^\times_{k,\fm} \cap N_{L/k}(L^\fm)]};
\end{equation}
here $N_{L/k}$ (resp. $N_{L_w/k_v}$) is the norm map for $L/k$ (resp. $L_w/k_v$), and $e_v$ is the ramification index of $v$ in $L$. In the first product, $w$ is an arbitrary place of $L$ above $v$ and the group $N_{L_w/k_v}(1+\fm \CO_w)$ is independent of the choice of $w$.

\begin{lem}\label{lem:unram_norm}
If $k$ is a finite unramified extension of $\BQ_2$, then for every $m\geq 2$ we have 
$N_{k/\BQ_2}(1+ 2^m \CO_k ) =1+2^m \BZ_2$, where $\CO_k$ is the ring of integers in $k$.
\end{lem}
\begin{proof}
Let $\mathrm{Tr}_{k/\BQ_2}$ be the trace map from $k$ to $\BQ_2$.
Let $\log$ be the usual $2$-adic logarithm map. 
It commutes with the Galois action. For $m\geq 2$, we then have a commutative diagram:
\[
\begin{tikzcd}
1+2^m \CO_k \ar[r,"\log"] \ar[d, "N_{k/\BQ_2}"'] & 2^m \CO_k \ar[d, "\mathrm{Tr}_{k/\BQ_2}"] \\
1+2^m \BZ_2 \ar[r, "\log"]  & 2^m \BZ_2.
\end{tikzcd}
\]
The two horizontal maps are isomorphisms as $m\geq 2$. The right vertical map is a surjection because $k/\BQ_2$ is unramified. The lemma then follows.
\end{proof}

\begin{proof}[Proof of Proposition~\ref{prop:ray1}]
 We apply \eqref{eq:chevalley} with respect to the cyclic extension $(K_n/K, \fm)$, with $\fm=(\fp^*)^2$. As 
 $\CO^\times_{K,\fm} =\{1\}$, then, writing $v=\fp^*$, the formula gives
 \begin{equation}
\# ((\fC_n)^{\Gal(K_n/K)}) =\#(\fC_0) \cdot [1+\fm \CO_{v}: N(1+\fm \CO_w)] =\#(\fC_0).
 \end{equation}
The second equality follows from Lemma~\ref{lem:unram_norm}. Using the fact that the class number of $K$ is odd, and that the units of $K$ are $\{\pm1\}$, it is readily verified that $\#(\fC_0)$ is odd. Therefore $\#( \fC_n)$ is odd by Nakayama's lemma, since $\Gal(K_n/K)$ is cyclic of order $2^n$. \end{proof}

\begin{proof}[Proof of Theorem~\ref{thm:main}]
Take $n$ to be an integer which is sufficiently large to ensure that there are precisely $s_\infty(R)$ primes of $K_n$ dividing $\sqrt{-q}R$. These primes are then all inert in $K_\infty$. 
To simplify notation, we shall write $s = s_\infty(R)$ in the rest of this proof. We have
\begin{equation}\label{e}
\sqrt{-q}R \CO_{K_n} = \fa_1\cdots \fa_s,
\end{equation}
where each $\fa_i$ is the power of some prime ideal of $K_n$, and $(\fa_i, \fa_j) = 1$ when $i \neq j$. Note that the exponents to which these prime ideals occur in \eqref{e} are all odd by the definition of $R$ and the fact $(\fq, R) = 1$. Write $h'$ for the order of $\fC_n$, which is odd by Proposition \ref{prop:ray1}.  Thus there exist $\alpha_i \in \CO_{K_n}$ such that 
\begin{equation}\label{eq:XF}
\fa^{h'}_i =(\alpha_i)  \quad\text{ and }   \quad \alpha_i \equiv 1\bmod (\fp^*)^2 \CO_{K_n} \quad \text{ for } 1\leq i\leq s-1.
\end{equation}
 We define
\begin{equation}\label{eq:alpha_rQ}
 \alpha_{s} = \frac{(-\sqrt{-q}R)^{h'}}{\alpha_1\cdots \alpha_{r-1}}.
\end{equation}
Note that, thanks to the oddness of $h'$ and our choice of the sign, we also have $\alpha_{s} \equiv 1\bmod  (\fp^*)^2 \CO_{K_n}$. Again because $h'$ is odd, it follows that, for $1\leq i\leq s$, the extension $K_n(\sqrt{\alpha_i})/K_n$ is quadratic and is ramified at the unique prime dividing $\alpha_i$.
 Moreover, this extension is unramified at the primes above $\fp^*$, since it is obtained by adjoining a root of
$x^2-x+\frac{1-\alpha_i}{4}$ and this polynomial is separable modulo each prime of $K_n$ above $\fp^*$. Define
\begin{equation}\label{eq:def_T}
T=K_\infty(\sqrt{\alpha_1},\cdots,  \sqrt{\alpha_{s}}).
\end{equation} 
Now $K_\infty(\sqrt{\alpha_i})/K_\infty$ is ramified at the unique prime dividing $\alpha_i$, while it is unramified at the primes dividing $\alpha_j$ for $j\neq i$, whence it follows easily that $[T:K_\infty]=2^s$, and
\begin{equation*}
\Gal(T/K_\infty)\cong (\BZ/2\BZ)^s.
\end{equation*}
 Thanks again to the oddness of $h'$, it follows from \eqref{eq:alpha_rQ} that
 \begin{equation}\label{eq:X(F)2}
 F_\infty \subset T  \quad \text{ and} \quad \Gal(T/F_\infty)  \cong (\BZ/2\BZ)^{s-1}.
\end{equation}
The extension $T/F_\infty$ is unramified at primes not dividing $\fp \fa_1 \cdots \fa_s$, since $T/K_\infty$ is. 
Moreover, $T/F_\infty$ is also unramified at the prime dividing $\fa_i$ for $1\leq i\leq s$, since from \eqref{eq:def_T} the ramification index of this prime in $T/K_\infty$ is clearly equal to $2$, but $F_\infty/K_\infty$ is also ramified at this prime by Lemma~\ref{lem:ram1}. This proves $T\subset M(F_\infty)$. Hence $X(F_\infty)$ has $(\BZ/2\BZ)^{s-1}$ as a quotient, and so its $\BZ_2$-rank must be at least $s-1$. Combining this with the upper bound given by Proposition~\ref{prop:Xbound}, the proof of assertion (i) of Theorem~\ref{thm:main} is complete.

Assume for the rest of the proof that $q\equiv 7\bmod 16$. By Proposition~\ref{prop:Xbound}, it suffices to construct $s+1$ independent quadratic extensions of $F'_\infty$ contained in $M(F'_\infty)$. Define
\begin{equation}\label{eq:def_T2}
T'= T(\sqrt{-1})= K_\infty(\sqrt{-1},\sqrt{\alpha_1},\cdots,  \sqrt{\alpha_{s}}).
\end{equation}
Note that $T'\supset F'_\infty$. Since $F'(\sqrt{-1})/F'$ is unramified at $\fp$, we have $F'_\infty(\sqrt{-1})\subset M(F'_\infty)$. Moreover, by the same argument as in the first paragraph, we have $T\subset M(F'_\infty)$. Thus $T'\subset M(F'_\infty)$, and so it remains to show that $[T':F'_\infty]=2^s$.
Note that $T/K_\infty$ is unramified at the unique prime above $\fp^*$, but $K_\infty(\sqrt{-1})/K_\infty$ is ramified at this prime. It follows that $T'\neq T$ whence $[T':K_\infty]=2^{s+1}$ and therefore $[T':F'_\infty]=2^s$. This proves $X(F'_\infty) \cong \BZ^s_2$, as required.

Finally, since $J/K$ is a biquadratic  extension, with $F,F'$ and $D$ being the nontrivial intermediate fields, it follows from Lemma~\ref{lem:prelim3} and Proposition~\ref{prop:Xbound} that $X(J_\infty)\cong \BZ^{2s-1}_2$. This completes the proof of Theorem~\ref{thm:main}. \end{proof}

\section{Equivalent formulation of Theorem \ref{thm:main} in terms of infinite descent on certain abelian varieties with complex multiplication}

The aim of this section is to briefly explain an equivalent formulation of Theorem \ref{thm:main} in terms of infinite descent on a family of quadratic twists of a certain higher dimensional
abelian variety with complex multiplication. We omit detailed proofs, and refer the reader to \cite{Gross1},\cite{CL} for further explanations of the background material. Fix an embedding of $K$ in $\BC$. Let $H = K(j(\CO_K))$ be the Hilbert class field of $K$, where $j$ is the classical modular function, and let $A/\BQ(j(\CO_K))$ be the Gross elliptic curve with complex multiplication by $\CO_K$ (see \cite{Gross1}, Chap. 5). Thus $A$ has minimal discriminant $(-q^3)$, and $A$ is isogenous over $H$ to all of its conjugates. Let 
$$
B=\mathrm{Res}^{H}_{K} A
$$
 be the $h$-dimensional abelian variety over $K$ which is the restriction of scalars of $A$ from $H$ to $K$.  Let $\msb = \End_K(B)$, and  $\mst = \msb \otimes \BQ$, so that $\mst$ is a CM field of degree $h$ over $K$. Then $\msb$ is an order in $\mst$, which is ramified over $\CO_K$ at precisely the primes dividing $h$ (see \cite{Gross1}, Theorem 15.2.5).  In particular, since $h$ is odd, the primes $\fp$, and $\fp^*$ are both unramified in $\mst$.  Now the torsion subgroup of $B(K)$ is 
$\CO_K/2\CO_K$, and the action of $\msb$ on this torsion subgroup  gives an $\CO_K$-algebra surjection from $\msb$ onto $\CO_K/2\CO_K$, whose kernel is the product of two conjugate primes $\fP, \fP^*$ of $\msb$ lying above $\fp, \fp^*$, respectively. These primes  are both unramified in $\msb$,  and have residue fields equal to the field $\BF_2$ with $2$ elements. Our equivalent formulation of Theorem \ref{thm:main} will be in terms of the arithmetic of the abelian variety $B^{(R)}/K$, which is the twist of $B$ by the quadratic extension $K(\sqrt{R})/K$. For each $n \geq 1$, let $B^{(R)}_{\fP^n}$ be the Galois module of $\fP^n$-division points on $B^{(R)}$. Then we have the following interpretation of the fields defined in $\S1$.

\begin{thm}\label{t1} We have 
\begin{align*}
 & F = K(B^{(R)}_{\fP^2}),  \quad F' = K(B^{(R)}_{{\fP^*}^2}),   \quad J = FF'= K(B^{(R)}_{\fP^2}, \sqrt{-1}) \\
 & F_\infty = K(B^{(R)}_{\fP^\infty}), \quad  F'_\infty = K(B^{(R)}_{{\fP^*}^\infty}),   \quad J_\infty = K( B^{(R)}_{\fP^\infty}, \sqrt{-1}).
\end{align*}
Moreover, $B^{(R)}$ has good reduction everywhere over $F$ and $F'$.
\end{thm}

\noindent We omit the detailed proofs, which are similar to those given in \cite{CL} for the abelian variety $B$ (see \cite[\S 2 and Lemma 7.11]{CL} and also see \cite[Lemma 2.1, 2.2]{CC}). 

\medskip

We next recall the definition of the $\fP^\infty$-Selmer group of $B^{(R)}$ over an algebraic extension of $K$. Take any non-zero endomorphism $\pi$ in $\msb$ such
that the ideal factorization of $\pi$ in the ring of integers of $\mst$ is equal to $\fP^r$ for some integer $r \geq 1$. Let $L$ be any algebraic extension of $K$. Then for each integer
$n \geq 1$, the group $\Sel_{\pi^n}(B^{(R)}/L)$ is defined by
$$
\Sel_{\pi^n}(B^{(R)}/L) = \Ker ( H^1(L, B^{(R)}_{\pi^n}) \to \prod_{v} H^1(L_v, B^{(R)})),
$$
where $v$ runs over all finite places of $L$, and $L_v$ is the compositum of the completions at $v$ of all finite extensions of $K$ contained in $L$. Passing to the inductive limit over
all $n \geq 1$, and noting that $B^{(R)}_{\pi^\infty} = B^{(R)}_{\fP^\infty}$ as Galois modules, we then define the Selmer group $\Sel_{\fP^\infty}(B^{(R)}/L)$ to be the inductive limit of the
$\Sel_{\pi^n}(B^{(R)}/L)$, so that:-
\[
\Sel_{\fP^\infty}(B^{(R)}/L) = \Ker( H^1(L, B^{(R)}_{\fP^\infty}) \to \prod_{v} H^1(L_v, B^{(R)})).
\]
Since $B^{(R)}$ has good reduction everywhere over $F$, and  $F_\infty = K(B^{(R)}_{\fP^\infty})$, it is then easily seen that we have (see \cite[Theorem 3.9]{CL} and \cite{C1}):-

\begin{thm}\label{t2} As Galois modules, we have
$$
\Sel_{\fP^\infty}(B^{(R)}/F_\infty) = \Hom(X(F_\infty), B^{(R)}_{\fP^\infty}),  \, \, \, \Sel_{\fP^\infty}(B^{(R)}/J_\infty) = \Hom(X(J_\infty), B^{(R)}_{\fP^\infty}),
$$
\end{thm}

 \noindent Noting that $B^{(R)}_{\fP^\infty} = \mst_\fP/\msb_\fP = \BQ_2/\BZ_2$ as abelian groups, we immediately obtain the following equivalent form of Theorem \ref{thm:main}.
 
 \begin{thm}\label{t3} As abelian groups, we have (i) $\Sel_{\fP^\infty}(B^{(R)}/F_\infty) = (\BQ_2/\BZ_2)^{s_\infty(R) - 1}$, and (ii) if $q\equiv 7\bmod 16$, $\Sel_{\fP^\infty}(B^{(R)}/J_\infty) = (\BQ_2/\BZ_2)^{2s_\infty(R) - 1}$
 for all $R$ satisfying the Twisting Hypothesis, and where $s_\infty(R)$ is the number of primes of $K_\infty$ dividing $\sqrt{-q}R$.
 \end{thm}
 
 \noindent  Let $\fF$ denote either of the fields $F$ or $J$. We recall that the Tate-Shafarevich group $\Sha(B^{(R)}/\fF_\infty)$ of $B^{(R)}/\fF_\infty$ is defined by
 $$
\Sha(B^{(R)}/\fF_\infty) =  \Ker(H^1(\fF_\infty, B^{(R)}) \to \prod_v H^1(\fF_{\infty, v}, B^{(R)})).
 $$
 It is a $\msb$-module and we write $\Sha(B^{(R)}/\fF_\infty)(\fP^\infty)$ for its $\fP$-primary subgroup.
 Note also that the group $B^{(R)}(\fF_\infty)$ of $\fF_\infty$-rational points of $B^{(R)}$ has a natural structure as an $\msb$-module, and we have the exact sequence
 \begin{equation}\label{4.1}
 0 \to B^{(R)}(\fF_\infty)\otimes_\msb(\mst_\fP/\msb_\fP) \to \Sel_{\fP^\infty}(B^{(R)}/\fF_\infty) \to \Sha(B^{(R)}/\fF_\infty)(\fP^\infty) \to 0.
 \end{equation}
Thus Theorem \ref{t3} can be reformulated as follows. Define $g_{\fF_\infty}(R)$ to be the $\mst$-dimension of $B^{(R)}(\fF_\infty)\otimes_{\BZ}\BQ$. Now $\Sha(B^{(R)}/\fF_\infty)(\fP^\infty)$ is a divisible group thanks to Theorem \ref{t2}, the fact that $X(\fF_\infty)$ is a free finitely generated $\BZ_2$-module, and the exact sequence \ref{4.1}, whence we define $e_{\fF_\infty}(R)$ to be its $\BZ_2$-corank. 

\begin{thm}\label{t4}  For all $R$ satisfying the Twisting Hypothesis, we have (i) $g_{F_\infty}(R)+ e_{F_\infty}(R)= s_\infty(R)-1$, and (ii) if $q\equiv 7\bmod 16$, $g_{J_\infty}(R)+ e_{J_\infty}(R) = 2s_\infty(R)-1$, where we recall that $g_{\fF_\infty}(R)$ is the $\mst$-dimension of $B^{(R)}(\fF_\infty)\otimes_{\BZ}\BQ$, $e_{\fF_\infty}(R)$ is the $\BZ_2$-corank of $\Sha(B^{(R)}/\fF_\infty)(\fP^\infty)$, and, as always, $s_\infty(R)$ is the number of primes of $K_\infty$ dividing $\sqrt{-q}R$.
\end{thm}

\noindent Since $s_\infty(R)$ is very easy to calculate explicitly, it seems to us that no simple explicit general result like Theorems \ref{t3} and \ref{t4} was known previously in the Iwasawa theory of elliptic curves. As an illustration, at the end of the next section we apply these theorems to one numerical example taken from \cite{CC}, with $s_\infty(R) = 4, g_{F_\infty}(R) = 2, e_{F_\infty}(R)=1.$

\section{Applications of Theorem \ref{thm:main}}  

In this last section, we discuss the situation when $R$ is a square free rational integer satisfying the Twisting Hypothesis, and having the property that $s_\infty(R) = 2$ (the case when $s_\infty(R) = 1$ is already discussed fully in \cite{CLL}). We also discuss one numerical example with $s_\infty(R) = 4$. The following lemma is clear from the remarks made immediately after the statement of Theorem \ref{thm:main}.

\begin{lem}\label{n1} Assume  that $R$ is a square free rational integer satisfying the Twisting Hypothesis. Then $s_\infty(R) = 2$ if and only if one of the following three cases hold:-
(i) we have $q \equiv 15 \mod 32$ and $R = 1$, (ii) we have $q \equiv 7 \mod 16$ and $R = r$, where $r$ is a prime with $r \equiv 5 \mod 8$, which is inert in $K$, or (iii) we have $q \equiv 7 \mod 16$ and $R = -r$, where $r$ is a prime with $r \equiv 3 \mod 8$, which is inert in $K$.
\end{lem}

By Theorem \ref{t3}, we know that, whenever $s_\infty(R) = 2$, we have $\Sel_{\fP^\infty}(B^{(R)}/F_\infty) = \BQ_2/\BZ_2$ as an abelian group. For each $n \geq 0$, recall that $K_n$ is the $n$-th layer of the $\BZ_2$-extension $K_\infty/K$. The cases $(i), (ii), (iii)$ of the next theorem refer to the three cases of Lemma \ref{n1}.

\begin{thm}\label{n2} In case (i), we have
$\Sel_{\fP^\infty}(B/K)$ is finite, $\Sel_{\fP^\infty}(B/K_n)$ has $\BZ_2$-corank at most $1$ for all $n \geq 1$, and  $\Sel_{\fP^\infty}(B/K_n)$ has $\BZ_2$-corank  equal to $1$ for some $n \geq 1$ if and only if $L(B/K_n, 1) = 0$. In case (ii), we have  $\Sel_{\fP^\infty}(B^{(R)}/K)$ is finite, $\Sel_{\fP^\infty}(B^{(R)}/K_n)$ has $\BZ_2$-corank at most 1 for all $n \geq 1$, and  $\Sel_{\fP^\infty}(B^{(R)}/K_n)$ has $\BZ_2$-corank equal to 1 for some $n \geq 1$ if and only if $L(B^{(R)}/K_n, 1) = 0$. In case (iii),  we have  $\Sel_{\fP^\infty}(B^{(R)}/K_n)$ has $\BZ_2$-corank equal to $1$ and $L(B^{(R)}/K_n, 1) = 0$ for all $n \geq 0$.
\end{thm}

We now outline the proof of this theorem, again referring to \cite{CL} for similar detailed arguments. We begin by recalling an elementary purely algebraic lemma, which holds for any $R$ in $\CO_K$ satisfying the Twisting Hypothesis. Put $\Gamma = \Gal(F_\infty/F)$, $\Gamma_n = \Gal(F_\infty/F_n)$, and fix any topological generator $\gamma$ of $\Gamma$. Let $\Lambda(\Gamma)$ be the Iwasawa algebra of $\Gamma$ with coefficients in $\BZ_2$. As usual, we identify $\Lambda(\Gamma)$ with the formal power series ring $\BZ_2[[T]]$ by mapping $\gamma$ to $1+T$. Since $X(F_\infty)$ is a finitely generated $\BZ_2$-module, it is clearly a torsion  $\Lambda(\Gamma)$-module, and we write $c_{F_\infty}(T)$ for its characteristic power series. By Theorem~\ref{t1}, we have a character $\kappa: \Gamma \to 1 + \fP^2 \msb_\fP = 1+4\BZ_2$ giving the action of $\Gamma$ on $B^{(R)}_{\fP^\infty}$ and we define $u=\kappa(\gamma)$.
%Recalling that $ \Sel_{\fP^\infty}(B^{(R)}/F_\infty) = \Hom(X(F_\infty), B^{(R)}_{\fP^\infty})$, we obtain:-

\begin{lem}\label{n3} For each $R$ in $\CO_K$ satisfying the Twisting Hypothesis, $(\Sel_{\fP^\infty}(B^{(R)}/F_\infty) )^{\Gamma_n}$ is infinite for some integer $n \geq 0$ if and only if $c_{F_\infty}(u\zeta-1) = 0$ for some $2^n$-th power root of unity $\zeta$.
\end{lem}

\begin{proof}
Put $\gamma_n = \gamma^{2^n}$, so that $\gamma_n$ is a topological generator of $\Gamma_n$, $\gamma_n = (1+T)^{2^n}$ and $\kappa(\gamma_n) = u^{2^n}$. Write $X=X(F_\infty)$ in this proof for simplicity. By Theorem~\ref{t2},  we have 
\begin{equation}\label{eq:char_pol1}
(\Sel_{\fP^\infty}(B^{(R)}/F_\infty))^{\Gamma_n} = \Hom(X/(\gamma_n-\kappa(\gamma_n))X, B^{(R)}_{\fP^\infty}).
\end{equation}
Thus, the left hand side is finite if and only if $X/(\gamma_n-\kappa(\gamma_n))X$ is finite. Note that 
\begin{equation}\label{eq:char_pol2}
\gamma_n-\kappa(\gamma_n) = (1+T)^{2^n}-u^{2^n} = \prod_{\zeta}(T-(u\zeta -1)),
\end{equation}
where $\zeta$ runs over all the $2^n$-th power of roots unity. Since $c_{F_\infty}(T)$ is the characteristic power series of $X$, it follows that  $X/(\gamma_n-\kappa(\gamma_n))X$ is finite if and only if $c_{F_\infty}(T)$ is coprime to $T-(u\zeta -1)$ for each $\zeta$. This proves Lemma~\ref{n3}.
\end{proof}

\noindent In addition, we have the following arithmetic lemma.  For each algebraic extension $L/K$, we define
$$
\Sel'_{\fP^\infty}(B^{(R)}/L) = \Ker ( H^1(L, B^{(R)}_{\fP^\infty}) \to \prod_{v\nmid \fp} H^1(L_v, B^{(R)})).
$$
Recall that $F_n = FK_n$ for $n \geq 0$.

\begin{lem}\label{n4} For each $R$ in $\CO_K$ satisfying the Twisting Hypothesis, and for all $n \geq 0$, $\Sel_{\fP^\infty}(B^{(R)}/K_n)$ has the same $\BZ_2$-corank as $\Sel'_{\fP^\infty}(B^{(R)}/F_n) = (\Sel_{\fP^\infty}(B^{(R)}/F_\infty) )^{\Gamma_n}$. 
\end{lem}

\begin{proof} Since $B^{(R)}$ has good reduction everywhere over $F$, entirely similar arguments to those given in \S3 of \cite{CL} show that
\begin{equation}\label{n10}
\Sel_{\fP^\infty}(B^{(R)}/F_\infty) = \Sel'_{\fP^\infty}(B^{(R)}/F_\infty),
\end{equation}
and that $\Delta$ acts trivially on  $\Sel'_{\fP^\infty}(B^{(R)}/F_\infty)$, where $\Delta = \Gal(F_\infty/K_\infty)$.  Using
again the fact that $B^{(R)}$ has good reduction everywhere over $F$, the same argument as in the proof of Proposition 3.11 of \cite{CL} shows that
\begin{equation}\label{n9}
(\Sel'_{\fP^\infty}(B^{(R)}/F_\infty))^{\Gamma_n} = \Sel'_{\fP^\infty}(B^{(R)}/F_n)
\end{equation}
for all $n \geq 0$. Thus to complete the proof, it suffices to show that the natural map
\begin{equation}\label{n5}
 \Sel_{\fP^\infty}(B^{(R)}/K_n) \to (\Sel'_{\fP^\infty}(B^{(R)}/F_n))^\Delta = \Sel'_{\fP^\infty}(B^{(R)}/F_n)  
\end{equation}
has finite kernel and cokernel for all $n \geq 0$. But this follows easily from the fact that $\Delta$ is of order 2, combined with a similar argument using Tate local duality for the abelian variety $B^{(R)}/K$, which has good reduction at the prime $\fp$, as that given in the proof of Theorem 3.12 of \cite{CL}.
\end{proof}

For each $n \geq 0$, let $L(B^{(R)}/F_n ,s)$ (resp. $L(B^{(R)}/K_n ,s)$) be the complex $L$-series of $B^{(R)}/F_n$ (resp.  $B^{(R)}/K_n$). Let $\chi_n$ be the non-trival
character of $\Gal(F_n/K_n)$, and write $L(B^{(R)}/K_n, \chi_n, s)$ for the twist of $L(B^{(R)}/K_n ,s)$ by $\chi_n$. Then we have
\begin{equation}\label{n7}
L(B^{(R)}/F_n,s) = L(B^{(R)}/K_n, s)L(B^{(R)}/K_n, \chi_n, s).
\end{equation}

\begin{thm}\label{n12} For all $R$ in $\CO_K$ satisfying the Twisting Hypothesis, and all $n \geq 0$, we have $L(B^{(R)}/K_n, \chi_n, 1) \neq 0$, and so
$L(B^{(R)}/F_n,1) \neq 0$ if and only if  $L(B^{(R)}/K_n, 1) \neq 0$.
\end{thm}

\begin{proof} This theorem, which is valid for all primes $q \equiv 7 \mod 8$ is an equivalent form of Theorem 1.5 of \cite{CL}. Put $\beta = \sqrt{-q}$. Now
$L(B^{(R)}/K_n, \chi_n, s)$ is the complex $L$-series of the twist of $B^{(R)}$ by the quadratic extension $F_n/K_n$. As $F_n = K_n(\sqrt{-\beta R})$, it follows
that the twist of $B^{(R)}$ by $F_n/K_n$ is just $B^{(-\beta)}/K_n$. Put $H_n = HK_n$, where we recall that $H$ is the Hilbert class field of $K$. Then
the complex $L$-series of $B^{(-\beta)}/K_n$ is the same as the complex $L$-series $L(A^{(-\beta)}/H_n, s)$ of $A^{(-\beta)}/H_n$, where $A/H$ is the basic Gross $\BQ$-curve defined over $H$. Since Theorem 1.5 of \cite{CL} asserts that $L(A^{(-\beta)}/H_n, 1) \neq 0$ for all $n \geq 0$, the proof is complete.
\end{proof}

The following deep theorem is a consequence of the main conjecture of Iwasawa theory for the $\BZ_2$-extension 
$F_\infty/F$. As above, $c_{F_\infty}(T)$ will denote a characteristic power series of the $\Gamma$-module $X(F_\infty)$.

\begin{thm}\label{n6} Let $n\geq 0$ be any non-negative integer, and let $R$ in $\CO_K$ satisfy the Twisting Hypothesis. Then $c_{F_\infty}(u\zeta-1) \neq 0$ for all $2^n$-th roots of unity $\zeta$ if and only if $L(B^{(R)}/F_n,1) \neq 0$.
\end{thm}
Combining this result with Theorem \ref{n12} and Lemmas \ref{n3} and \ref{n4}, we immediately obtain 
\begin{cor}\label{n13} For all $n \geq 0$, we have $\Sel_{\fP^\infty}(B^{(R)}/K_n)$ is finite if and only if $L(B^{(R)}/K_n, 1) \neq 0$. 
\end{cor}
\noindent Of course, this is not the place to give a detailed proof of this main conjecture. We simply say that similar methods, using the Euler system of elliptic units, to those used  to prove Theorem 7.13 in \S7 of \cite{CL}, can be generalized to prove the desired main conjecture for $B^{(R)}/F_\infty$, from which the above theorem follows as an immediate corollary.

\medskip
\begin{proof} We now prove Theorem \ref{n2}. We first suppose we are in case (i), so that $q \equiv 15 \mod 32$ and $R = 1$. By a theorem of Rohrlich \cite{R}, we have $L(B/K, 1) = L(A/H, 1) \neq 0$. Thus, by Corollary \ref{n13}, $\Sel_{\fP^\infty}(B/K)$ is finite, and, for all $n \geq 1$, $\Sel_{\fP^\infty}(B/K_n)$ is finite if and only if $L(B/K_n, 1) \neq 0$. Moreover,  as $s_\infty(R) =2$, it follows from Theorem \ref{t3} that $\Sel_{\fP^\infty}(B/F_\infty) = \BQ_2/\BZ_2$ as an abelian group, and so Lemma \ref{n4} implies that $\Sel_{\fP^\infty}(B/K_n)$ has $\BZ_2$-corank at most 1 for all $n \geq 1$. This completes the proof of case (i). Next suppose we are in the case (ii), so that $q \equiv 7 \mod 16$ and $R=r$, with $r \equiv 5 \mod 8$ and $r$ inert in $K$. Then Theorem 1.3 of \cite{CL} shows that $L(B^{(R)}/K, 1) = L(A^{(R)}/H, 1) \neq 0$. Again invoking Corollary \ref{n13}, it follows that $\Sel_{\fP^\infty}(B^{(R)}/K)$ is finite, and, for all $n \geq 1$, $\Sel_{\fP^\infty}(B^{(R)}/K_n)$ is finite if and only if $L(B^{(R)}/K_n, 1) \neq 0$.  But  $\Sel_{\fP^\infty}(B^{(R)}/F_\infty) = \BQ_2/\BZ_2$ as an abelian group, whence, again by Lemma \ref{n4}, $\Sel_{\fP^\infty}(B^{(R)}/K_n)$ has $\BZ_2$-crank at most $1$ for all $n \geq 1$, and the proof of case (ii) is complete. Finally, suppose we are in case (iii), so that $q \equiv 7 \mod 16$ and $R= -r$, with $r\equiv 3 \mod 8$ and $r$ inert in $K$. 
Now $B^{(R)}$ is in fact defined over $\BQ$,
and $ L(B^{(R)}/K, s) =  L(B^{(R)}/\BQ, s)^2$, where $ L(B^{(R)}/\BQ, s)$ is the complex $L$-series of $B^{(R)}/\BQ$ (see \S18 of \cite{Gross1}). Write
$$
\Phi(B^{(R)}/\BQ, s)) = (-qr/(2\pi))^{hs}\Gamma(s)^hL(B^{(R)}/\BQ, s),
$$
where we recall that $h$ is the class number of $K$. Now, since $h$ is odd and $2$ splits in $K$, Theorem 19.1.1 of \cite{Gross1} shows that $\Phi(B^{(R)}/\BQ, s)$ satisfies
the functional equation
$$
\Phi(B^{(R)}/\BQ, s) = - \Phi(B^{(R)}/\BQ, 2-s).
$$
Hence $\Phi(B^{(R)}/\BQ, s))$ must have a zero of odd order at $s=1$, and so $L(B^{(R)}/\BQ, 1) = 0$, whence also $L(B^{(R)}/K, 1) = 0$, which, in turn,
clearly implies that $L(B^{(R)}/F, 1) = 0$, and $L(B^{(R)}/K_n, 1) = 0$ for all $n \geq 0$. Now  $\Sel_{\fP^\infty}(B^{(R)}/F_\infty)  = \BQ_2/\BZ_2$ as an abelian group, and 
$\Sel_{\fP^\infty}(B^{(R)}/F_\infty)^\Gamma$ is infinite because $L(B^{(R)}/F, 1) = 0$. Hence, since the only infinite subgroup
of $\BQ_2/\BZ_2$ is the whole group, it follows that
\begin{equation}\label{n8}
\Sel_{\fP^\infty}(B^{(R)}/F_\infty)^\Gamma = \Sel_{\fP^\infty}(B^{(R)}/F_\infty) = \BQ_2/\BZ_2.
\end{equation}
Thus, by Lemma \ref{n4}, for all $n \geq 0$, we have $\Sel'_{\fP^\infty}(B^{(R)}/F_n) = \BQ_2/\BZ_2$, whence $\Sel_{\fP^\infty}(B^{(R)}/K_n)$ has $\BZ_2$-corank 1. This completes
the proof of case (iii), and so of the whole theorem.
\end{proof}

\medskip
 
We end by discussing one numerical example with $s_\infty(R) = 4$, which arose in \cite{CC}. Let $q = 7$, so that $B$ is the elliptic curve 
$X_0(49)$, with equation
\begin{equation}\label{4.2}
y^2 + xy = x^3 - x^2 - 2x - 1.
\end{equation}
Take $R = 741 = 3.13.19$, and define $E = X_0(49)^{(R)}$, whence $s_\infty(R) = 4$. By Theorem \ref{t4}, we then have $g_{F_\infty}(R) + e_{F_\infty}(R) = 3$. It is shown in \cite{CC} that $E(K)\otimes \BQ$ has
dimension $2$ over $K$ and $\Sha(E/K)(2)=0$. We shall prove the following result. We recall that in this case $F_\infty = K(E_{\fp^\infty})$, where $K = \BQ(\sqrt{-7})$, and $F_n = K(E_{\fp^{n+2}})$ for all $n \geq 0$.

\begin{prop} Let $E = X_0(49)^{(R)}$, with $R = 741$. Then $g_{F_\infty} =2$, and $e_{F_\infty}(R) = 1$, so that $E(F_\infty)\otimes_{\BZ}\BQ$ has $K$-dimension equal to $2$, and $\Sha(E/F_\infty)(\fp^\infty) = \BQ_2/\BZ_2$. Moreover, for all integers $n \geq 0$, we have that $E(F_n)\otimes_{\BZ}\BQ$ has $K$-dimension equal to $2$ and $\Sha(E/F_n)(\fp^\infty)$ is finite.
\end{prop}

\begin{proof}There are two ingredients to the proof. The first is a MAGMA calculation. Let $C/K$ be the twist of $E/K$ by the quadratic extension $K_1/K$. Then a  MAGMA computation shows that $C(K)\otimes \BQ$ has dimension $0$ and $\Sha(C/K)(2)=0$. It follows easily that $E(K_1)\otimes \BQ$ has $K$-dimension $2$ and $\Sha(E/K_1)(\fp^\infty)$ is finite, whence $\Sel_{\fp^\infty}(E/K_n)$ has corank $2$ for $n=1$. We now recall that (see \cite[Theorem 2.13]{CC}), for all $n\geq 0$ both the $K$-dimensions of $E(K_n)\otimes \BQ$ and the corank of  $\Sha(E/K_n)(\fp^\infty)$ do not change if we replace $K_n$ with $F_n$. Thus the MAGMA calculation proves that $\Sel_{\fp^\infty}(E/F_1)$ as well as $\Sel'_{\fp^\infty}(E/F_1)$ has $\BZ_2$-corank equal to 2 by Lemma~\ref{n4}. The second ingredient in the the proof is a simple remark on the characteristic power series $c_{F_\infty}(T)$ of $X(F_\infty)$. By the Weierstrass Preparation Theorem, we may assume that $c_{F_\infty}(T)$ is a monic polynomial of degree 3, because $X(F_\infty)$ is a free $\BZ_2$-module of rank 3. %By Lemma~\ref{n4},
%For $n \geq 0$, let $\fp_n$ be the unique prime of $F_n$ above $\fp$, and define
%\[
%\Sel'(E/F_n) = \Ker(H^1(F_n, E_{\fp^\infty}) \to \prod_{v \neq \fp_n}H^1(F_{n,v}, E)(\fp^\infty)). 
%\]
%$\Sel_{\fp^\infty}(E/F_n)$ and $\Sel'(E/F_n)$ have the same $\BZ_2$-corank for all $n \geq 0$, and, writing
%$\Gamma_n = \Gal(F_\infty/F_n)$, we have (see Prop. 2.11 of \cite{CC})
%\begin{equation}\label{4.3}
%\Sel'(E/F_n) = \Sel_{\fp^\infty}(E/F_\infty)^{\Gamma_n}.
%\end{equation}
%Let $\kappa: \Gamma \to 1+\fp^2$ be the character giving the action of $\Gamma$ on $E_{\fp^\infty}$, and put $u = \kappa(\gamma)$, where $\gamma$ is our fixed topological
%generator of $\Gamma$.
The characteristic power series of the Pontrjagin dual of $\Sel_{\fp^\infty}(E/F_\infty) = \Hom(X(F_\infty), E_{\fp^\infty})$ is given by $c_{F_\infty}(u(1+T) -1)$. Moreover, since $\Sel'_{\fp^\infty}(E/F)$ has $\BZ_2$-corank equal to 2, we conclude that $(T- (u-1))^2$ must divide $c_{F_\infty}(T)$. Suppose now that there exists an integer $n > 0$ such that $\Sel'_{\fp^\infty}(E/F_n)$ has $\BZ_2$-corank at least 3, and choose the smallest integer $n$ with this property. Then it follows from Lemma~\ref{n4} that $c_{F_\infty}(T)$ must also be dvisible by $T - (u\zeta - 1)$ for some primitive $2^n$-th root of unity $\zeta$. But, as $c_{F_\infty}(T)$ is a polynomial with coefficients in $\BZ_2$,
it would then be divisible by the product of all $T - (u\alpha -1)$, where $\alpha$ runs over all primitive $2^n$-th roots of unity. Since this product is an irreducible polynomial of degree $2^{n-1}$
with coefficients in $\BZ_2$, and $c_{F_\infty}(T)$ has degree 3, we must have $2^{n-1} = 1$, or equivalently $n = 1$. But our MAGMA calculation has 
already shown that $\Sel'(E/F_1)$ has $\BZ_2$-corank 2, and the proof is complete. \end{proof}

\medskip

\medskip

\noindent 
CAS Wu Wen-Tsun Key Laboratory of Mathematics, University of Science and Technology of China, \\
Hefei, Anhui 230026, China. \\
{\it lijn@ustc.edu.cn}

\end{document}